\documentclass[12pt]{amsart}
\usepackage{amsmath,latexsym, graphicx}
\usepackage[psamsfonts]{amssymb}
\usepackage{times}
\usepackage[mathcal]{euscript}
\usepackage{color}

\usepackage{amsfonts}
\usepackage{amssymb}
\usepackage{amsmath}
\usepackage{amsthm}
\usepackage{bm}
\usepackage[T1]{fontenc}
\usepackage[english]{babel}
\usepackage[bookmarks]{hyperref}
\numberwithin{equation}{section} \textwidth=140mm \textheight=200mm
\newcommand{\bbR}{\mathbb R}

\newcommand{\bbZ}{\mathbb Z}

\newcommand{\bbC}{\mathbb C}
\newcommand{\bbN}{\mathbb N}

\renewcommand{\epsilon}{\varepsilon}

\renewcommand{\epsilon}{\varepsilon}

\newcommand{\be}{\begin{equation}}
\newcommand{\ee}{\end{equation}}

\newcommand{\graph}{\mathrm{Graph}}

\newcommand{\R}{\mathbb{R}}

\newcommand{\cI}{{\mathcal I}}
\newcommand{\cJ}{{\mathcal J}}

\newcommand{\cQ}{{\mathcal Q}}
\newcommand{\cR}{{\mathcal R}}

\renewcommand{\Im}{{\ensuremath{\mathrm{Im}}}}

\newcommand{\dist}{{\ensuremath{\mathrm{dist}}}}

\newcommand{\tr}{\mathrm{tr}}

\renewcommand{\det}{\mathop{\mathrm{det}}}

\newcommand{\Dom}{\mathrm{Dom}}

{\bf}{\it}
\newtheorem{theorem}{Theorem}[section]
\newtheorem{lemma}[theorem]{Lemma}
\newtheorem{corollary}[theorem]{Corollary}

\newtheorem{definition}[theorem]{Definition}
\newtheorem{proposition}[theorem]{Proposition}

\newtheorem{remark}[theorem]{Remark}

\begin{document}

\small
\title[Virtual States  of a Jacobi  operator]{  Threshold 
Virtual States  of a Jacobi  operator}

\author{Saidakhmat N.~Lakaev, Konstantin  A. Makarov}

\address{Samarkand State University, 140104, Samarkand, Uzbekistan}
\email{slakaev@mail.ru}

\address{University of Missouri, Columbia, MO 63211, USA}
\email{makarovk@missouri.edu}

\begin{abstract}
 We prove that the set of parameters for which a virtual level appears at the edge of the continuous spectrum of a Jacobi matrix with a finite-rank diagonal perturbation constitutes an algebraic variety of codimension one. This variety partitions the parameter space into connected components, with their number determined by the size of the perturbation support.
We also  reveal a hierarchical structure underlying these critical varieties as the rank of the perturbation increases.
\end{abstract}

\maketitle
\section{Introduction}

We consider a discrete Schr\"odinger operator on the semi-infinite lattice $\mathbb{N}$,
 defined as a finite-rank diagonal perturbation of the Jacobi operator $J$,
 \begin{equation}\label{J}J=\begin{pmatrix}
2&-\sqrt{2}&0&0&\dots
\\-\sqrt{2}&2&-1&0&\dots
\\
0&-1&2&-1&\dots
\\\vdots&\vdots&\vdots&\vdots&\ddots
\end{pmatrix},
\end{equation}
 which models two spinless bosons with vanishing center-of-mass momentum (see, e.g., \cite{Valiente}). Small perturbations of $J$ can induce the appearance or disappearance of eigenvalues at the edges of the continuous spectrum of $J$, corresponding to the formation or loss of virtual states. The emergence of virtual levels at a threshold of the continuous spectrum is equivalent to the vanishing of the corresponding Jost function, which, for finite-rank perturbations, is a polynomial in several variables. The  study of the  geometry of the real affine algebraic variety defined by this Jost polynomial -- part of Hilbert's sixteenth problem -- forms the main focus of this work.
 We show that the left-threshold Jost function $C_n$ associated with the Jacobi operator
$J_n=J+P_nVP_n,$
where
 $P_n$ is the orthogonal projection onto the linear   span  of the first $n$  elements of the standard 
 basis $\{\delta_n\}_{n\in \bbN} $  in the Hilbert space $\ell^2(\mathbb{N})$, admits the representation 
$C_n = Q_n - Q_{n-1}$, 
where the polynomials $Q_k$ satisfy a three-term recurrence relation determined by the  magnitude of the perturbation  potential $V$  at the lattice  cells. This leads naturally to a hierarchy of affine varieties and  maps
\be\label{hier}
\dots \to \mathbf{V}(Q_n) \to \Dom(\Phi_n) \to \graph(\Phi_n) \to \mathbf{V}(Q_{n+1}) \to \dots,
\ee
generated by the rational function
  $$\Phi_n=Q_{n-1}/Q_n-2.$$ 
This framework allows for an
 inductive construction and a more detailed description of  the varieties ${\bf V}(Q_n)$ and ${\bf V}(C_n)$, respectively, as the $n$ increases The right-threshold case follows analogously via a symmetry principle for the Jost function.

Our approach yields a unified, explicit framework for tracking virtual states and their hierarchy in finite-rank perturbations of discrete Schr\"odinger operators: the variety $\mathbf{V}(C_n)$ partitions $\mathbb{R}^n$ into $n+1$ connected components, with each boundary crossing producing exactly one new bound state when the finite-rank potential is supported on $[1,\dots,n]$.

For background material on Jacobi operators,  Jost solutions and Jost functions and the related topics we refer to  the books \cite{Simon05I,Simon05II,Simon10,Teschl} and recent publications \cite{Yafaev17,Yafaev22}.

The paper organized as follows.

In Section 2, we review key properties of the perturbation determinant and the associated Jost functions for finite-rank perturbations $V$ of the Jacobi operator $J$. We derive explicit representations for their meromorphic continuation to the Riemann surface $\mathcal{R}$, which can be viewed as two copies of the complex plane joined along the cut corresponding to the absolutely continuous spectrum of $J$. Using a local parameter on $\mathcal{R}$ defined by the standard dispersion relation  (see eq. \eqref{guk}), we show that the Jost function becomes a polynomial, while the perturbation determinant becomes a rational function with at most two simple poles located at the spectral edges. The threshold behavior of the Jost function is then identified through the asymptotics of the perturbation determinant (see Corollary \ref{inz}).

Section 3 introduces the notion of a critical operator: $J+V$ is critical at a spectral threshold if an arbitrarily small perturbation of $V$ produces a new eigenvalue beyond that threshold. We prove that criticality at the left (resp. right) threshold occurs precisely when the corresponding threshold Jost function vanishes, and in this case the associated Jost solution generates a virtual state. The description for the right threshold follows analogously.

In Section 4, we analyze the geometry of the hierarchy \eqref{hier}, showing that the affine variety $\mathbf{V}(Q_n)$ associated with the polynomials $Q_n$ partitions $\mathbb{R}^n$ into $n+1$ connected components (Theorem \ref{decomm}, Corollary \ref{nuliQ}). As a consequence, the nodal surfaces of the threshold Jost function $C_n$ divide $\mathbb{R}^n$ into $n+1$ unbounded regions, $({\bf V}(C_n))^c = \bigcup_{k=0}^n G_k^{(n)}$ (Theorem \ref{main00}), and we present an efficient inductive algorithm for constructing these components.

Section 5 is auxiliary, providing results on the discrete spectrum of $J + tV$ in the large-coupling limit for the case where $V$ is the difference of two orthogonal projections.

In Section 6, we establish a precise spectral characterization: for a finite-rank potential $V$ supported on $[1,\dots,n]$, the operator $J+V$ has exactly $k$ eigenvalues below the lower threshold and $\ell$ eigenvalues above the upper threshold $(0 \le k+\ell \le n)$ if and only if the vector of potential strengths belongs to the intersection of $G_k^{(n)}$ with the spatial inversion of $G_\ell^{(n)}$, provided no threshold virtual states are present. Refinements in the presence of virtual levels are also discussed. 
 For a general discussion of the concept of virtual states in both the continuous and lattice cases, we refer, for example, to \cite{ALMM:2006,Yafaev79,Yafaev82}.

Finally, Appendix A provides concise self-contained proofs of mostly known results collected in Propositions \ref{propjost} and \ref{propdet}, which may be of independent interest. Proposition \ref{pams}, drawn from subspace perturbation theory, is included for completeness and ease of reference.

\section{The Jost function and perturbation determinant}
Consider a Jacoby operator $\cJ$  in $\ell^2(\bbN) $
 given by the matrix 
\be\label{JJ}
\mathcal{J}=\begin{pmatrix}
b_1&a_1&0&0&0&\dots\\
a_1&b_2&a_2&0&0&\dots\\
0&a_2&b_3&a_3&0&\dots\\
\dots&\dots&\dots&\dots&\dots&\dots
\end{pmatrix}
\ee
with
$$
\lim_{n\to\infty} a_n=-1\quad \text{and }\quad  \lim_{n\to\infty} b_n=-2.$$

 Recall (see, e.g., \cite{Killip, Yafaev22}) that
under the short range assumption 
$$
\sum_{n=1}^\infty|b_n-2|<\infty,
$$
a {\it Jost solution }  associated with $\cJ$ is defined as a  solution of the system of equations
\begin{equation}\label{JJJ}
a_{n-1}j_{n-1}+b_nj_n+a_nj_{n+1}=zj_n, \quad n=1,2,\dots
\end{equation}
for a sequence $\{j_n\}_{n=0}^\infty$
which is   asymptotic to   $\Theta^n(z)$ in the sense that 
$$
\lim_{n\to \infty} \Theta^{-n}(z) j_n(z)=1.
$$
Here $\Theta=\Theta(z)$  is a root of the  (dispersion) equation 
 \begin{equation}\label{guk}
\Theta+\frac{1}{\Theta}=2-z, \quad z\in \bbC\setminus [0,4],
\end{equation}
such that 
$$
|\Theta(z)|<1.
$$

The coefficient $a_0\ne 0$ in \eqref{JJJ} can  be  chosen at our convenience and  we set 
$$
 a_0=-\sqrt{2}.
 $$

Recall that  the value $j_0(z)$ of the Jost solution ``at zero'' is called the {\it Jost function}.
 
For a deeper discussion of the Jost functions/solutions and recent developments we refer to  \cite{DamanikI,DamanikII,Jensen,Yafaev79,Yafaev17,Yafaev18}.

For the free Jacobi matrix  given by \eqref{J}, in \eqref{JJJ} 
 we choose
 \footnote{ We note that the operator defined by the Jacobi matrix $J$ in the standard basis of the space $\ell^2(\bbN)$ is unitarily equivalent to the restriction of the discrete Laplacian  in  $\ell^2(\bbZ)$ to its invariant subspace of symmetric sequences.}
 $$
 b_k=2, \quad k=1, 2, \dots , 
 $$
and 
 $$
 a_1=-\sqrt{2}, \quad a_k=-1,\quad k>1.
 $$

Let $V$ be an arbitrary diagonal matrix
\begin{equation}\label{V}
V=\text{diag}[{v_1},v_2,\dots], \quad v_k\in \bbR, \quad k=1,2,\dots .
\end{equation}
Denote by  $J_n$ the  finite-rank diagonal perturbation of $J$ given by
\begin{equation}\label{jac}J_n=J+P_nVP_n,
\end{equation}
where
 $P_n$ is the orthogonal projection onto the linear   span  of the first $n$  elements of the standard 
 basis $\{\delta_n\}_{n\in \bbN} $  in the Hilbert space $\ell^2(\mathbb{N})$.

 \begin{proposition}\label{propjost}
 The Jost function $j_0^{(n)}(z)$ associated with the Jacoby operator $J_n$ \eqref{jac} can  explicitly be  represented 
 as \be\label{jn}
 j_0^{(n)}(z)=\frac12\Theta^n(q_n-\Theta q_{n-1}), \quad n\ge 1, \quad z\in \bbC\setminus[0,4],
 \ee
 where $\Theta=\Theta(z)$   is given by \eqref{guk}.
 Here  $q_n$ are  polynomials  in $z$ and $v_1, \dots, v_n$ satisfying the recurrence relations 
 \be\label{recq}
 q_{n+1}=(2-z+v_{n+1})q_{n}-q_{n-1}, \quad n\ge 1,
 \ee
 with the initial data
 \be\label{indat}
 q_0=2
 \quad \text{and}\quad 
 q_1=2-z+v_1.
 \ee
 
 Moreover,  the Jost function  $ j_{0}^{(0)}(z)$ associated with  the unperturbed Jacobi matrix $J_0=J$ is given by 
 \be\label{j00}
 j_{0}^{(0)}(z)=\frac{ 1-\Theta(z)^2}{2}.
 \ee
 
 \end{proposition}
 
 \begin{proof} See Appendix  \ref{A1}.
 \end{proof}

\begin{remark} Notice that the chosen contracting branch of the function $\Theta(z)$ conformally maps the complement
$(\mathbb{C} \cup \{\infty\}) \setminus [0,4]$
onto the interior of the unit disk. The points $\Theta = \pm 1$ correspond to the threshold values of the spectral parameter $z$: the lower edge of the continuous spectrum at $z = 0$ corresponds to $\Theta = +1$, while the upper edge at $z = 4 $ corresponds to $\Theta = -1$.

It is also worth noting that the functions $z$ and $\Theta(z)$ admit a  meromorphic continuation to a Riemann surface $\mathcal{R}$, obtained as the double of the complex plane with the cut
$\mathbb{C} \cup \{\infty\} \setminus [0,4]$,
by gluing two copies of the plane crosswise along the cut. On this Riemann surface $\mathcal{R}$, the meromorphic function $\Theta$ has a single zero on the first sheet and a single pole on the second sheet, whereas the function $z$ has one zero and one pole on each sheet. Moreover, the meromorphic  function   $z$ is a rational function of $\Theta$, and the corresponding functional relation \eqref{guk} is commonly referred to as the dispersion relation.

 In this context, we remark that Proposition \ref{propjost} allows us to regard the Jost function
$j_0^{(n)}(\cdot)$, associated with the Jacobi operator $J_n$, as a meromorphic function on the entire Riemann surface $\mathcal{R}$.
Taking into account the conformal change of variables \eqref{guk}, instead of working with the function $j_0^{(n)}(z)$ on $\mathbb{C}\setminus[0,4]$, is  convenient to  consider the Jost function $j_0^{(n)}(\Theta)$ as a function of the local parameter $\Theta\in \bbC$.
\end{remark}

 As a corollary of Proposition \ref{propjost} , we present a statement which asserts that the Jost function, as a function of the local parameter $\Theta$ on the Riemann surface $\cR$, is a polynomial.

 \begin{corollary}\label{teit}

 The Jost function $j_0^{(n)}(\Theta)$, associated with the Jacobi operator $J_n$, is a polynomial in the local parameter
 $  \Theta $ of degree $ 2n-1$  when $ v_1, \dots, v_n $  are fixed, and a polynomial in $ v_1, \dots, v_n $ of degree $ n $ when $ \Theta$  is fixed. It admits the representation
\be\label{figa}j_0^{(n)}(\Theta) = \frac{\cQ_n - \Theta^2 \cQ_{n-1}}{2},\quad \Theta\in \bbC,
\ee
where $ \cQ_n$  are polynomials in $ v_1, \dots, v_n$  of degree $ n $  satisfying the recurrence relation
$$\cQ_n = (1 + v_n \Theta + \Theta^2) \cQ_{n-1} - \Theta^2 \cQ_{n-2},
$$
with initial conditions
$$
\cQ_0 = 2, \quad \cQ_1 = 1 + v_1 \Theta + \Theta^2.
$$

Moreover, the following  Jost function  Symmetry Principle holds:
\begin{equation}\label{symm}
j_0^{(n)}(V;\Theta) = j_0^{(n)}(-V; -\Theta),
\end{equation}
where we have used the notation  $j_0^{(n)}(V;\Theta)$ for the Jost function  to explicitly indicate the dependence of the Jost function on the interaction potential $V$.
 
In particular, 
\be\label{por}
j_0^{(1)}(\pm1)=\pm\frac12v_1.
\ee

\end{corollary}\label{symremark}

\begin{proof}
The symmetry relations
\be\label{qsi}
\cQ_n(V,\Theta)=\cQ_n(-V,-\Theta)
\ee
obviously hold for $n=0,1$. By induction, using the recurrence relation for the polynomials $\cQ_n$, we see that \eqref{qsi}
 also holds for all $n \ge 2$
for the combination $1+v_n\Theta+\Theta^2$ is unchanged under the simultaneous transformation $(V,\Theta)\mapsto(-V,-\Theta)$. Combining \eqref{qsi} and \eqref{figa} implies \eqref{symm}.

To prove \eqref{por}, note that
\begin{align*}
j_0^{(n)}(\Theta)
&=\frac{\cQ_1-\Theta^2\cQ_0}{2}
=\frac{1+v_1\Theta+\Theta^2-2\Theta^2}{2}
\\
&=\frac{1+v_1\Theta-\Theta^2}{2},
\end{align*}
which yields \eqref{por}.
\end{proof}

\begin{remark}
Notice that  the indicated  Jost function  Symmetry Principle reflects  the equivalence of forward and backward lattice propagation and guarantees that the analytic structure of the Jost function is compatible with the two-sheeted Riemann surface of the dispersion relation \eqref{guk}.
\end{remark}

 Along with the Jost function, the perturbation determinant can be considered as a function of the  local parameter $\Theta$ on the two-sheeted Riemann surface $\cR$. In this parametrization, the determinant  becomes a rational function 
in $\Theta$,  and, in the generic case, has two simple poles at the points $+1$ and $-1$, corresponding to the threshold values $\lambda=0$ and $\lambda=4$ of the spectral parameter.

Note, however, that for special values of the interaction parameters, the corresponding pole-type singularities may be absent.

We present the corresponding result and provide a short proof in the Appendix.

\begin{proposition}\label{propdet} The perturbation determinant $
{\det}_{J_n/J}$ associated with the pair $(J_n,J)$ of Jacobi matrices  \eqref{jac} admits the representation
\be\label{vita}
{\det}_{J_n/J}\left (2-\Theta-\frac1\Theta\right )=\frac{j_0^{(n)}(\Theta)}{j_0^{(0)}(\Theta)},\quad \Theta\in \bbC\setminus\{-1,1\},
\ee
 where $ j_0^{(0)}$ is given by 
 $$ j_0^{(0)}(\Theta)=\frac{ 1-\Theta^2}{2}.
 $$
\end{proposition}
 \begin{proof} See Appendix  \ref{A2}.
 \end{proof}

We also provide convenient expressions for the Jost function at the edges of the continuous spectrum via the threshold asymptotics of the perturbation determinant considered as a function of the spectral parameter.
(Notice that in multidimensional problems, where the machinery of the Jost function is no longer available, the study of relevant threshold asymptotics of the perturbation determinant ${\det}_{J+V/J}(z)$  plays a fundamental role in understanding the mechanism of the emergence of virtual levels
 (see, e.g., \cite{LakMot}).)

 \begin{corollary}\label{inz} The following representations 
\begin{equation}\label{sleva}
j_0^{(n)}(1)=\lim_{z\uparrow 0}\sqrt{-z}\,{\det}_{J_n/J}(z)
\end{equation}
and 
\begin{equation}\label{sprava}
j_0^{(n)}(-1)=\lim_{z\downarrow 4}\sqrt{z-4}\,{\det}_{J_n/J}(z)
\end{equation}
hold.

 In particular, the perturbation determinant  ${\det}_{J_n/J}(z)$ is bounded in a neighborhood of the left threshold $z=0$ if and only if  the threshold Jost function $j_0^{(n)}(1)$ vanished at the threshold.
 Analogously,  the perturbation determinant  ${\det}_{J_n/J}(z)$ is bounded in a neighborhood of the right threshold $z=4$ if and only if  the threshold Jost function $j_0^{(n)}(-1)$ vanished at the threshold.
 
\end{corollary}

\begin{proof} Since from \eqref{guk} it follows that 
$$
-z=\Theta+\frac1\Theta-2=\left (\sqrt{\Theta}-\frac{1}{\sqrt{\Theta}}\right )^2, \quad z<0, \quad 0<\Theta<1,
$$
we have
$$
\lim_{z\uparrow 0}\sqrt{-z}\,{\det}_{J_n/J}(z(\Theta))= 
\lim_{\Theta\uparrow 1}\frac{1-\Theta}{\sqrt{\Theta}}\frac{2}{1-\Theta^2}\cdot j_0^{(n)}(\Theta)=
j_0^{(n)}(1).
$$

To prove \eqref{sleva}, we use 
$$z-4=\left (\sqrt{-\Theta}+\frac{1}{\sqrt{-\Theta}}\right )^2, \quad z>4, \quad -1<\Theta<0,
$$
and  see that 
\begin{align*}
\lim_{z\downarrow 4}\sqrt{z-4}\,{\det}_{J_n/J}(z(\Theta))&=\lim_{\Theta\downarrow -1}
\frac{1+\Theta}{\sqrt{-\Theta}}\frac{2}{1-\Theta^2}\cdot j_0^{(n)}(\Theta)=
j_0^{(n)}(-1).
\end{align*}

From Proposition \eqref{propdet} it follows that the condition  $j_0^{(n)}(1)=0$  (resp. $j_0^{(n)}(-1)=0$)
means that the perturbation determinant 
$${\det}_{J_n/J}\left (z\right )={\det}_{J_n/J}\left (2-\Theta-\frac1\Theta\right ), \quad \Theta\in \bbC,$$ is analytic in a neighborhood of the point $\Theta=1$  (resp. $\Theta=-1$), and therefore  ${\det}_{J_n/J}(z)$, as a function of the spectral parameter $z$,  is bounded in a neighborhood of the  threshold $z=0$ (resp. $z=4$), on the first sheet of the double, completing the proof.  

\end{proof}

\section{Critical Operators and the threshold virtual states}

To better understand  the  threshold phenomena associated with the birth and annihilation of eigenvalues, introduce the concept of a critical operator.

\begin{definition} Suppose that $V=V^*$ is a compact operator in  $\ell^2(\bbN )$.
We say that the operator $J+V$   is critical  
 at its lower  threshold  $\lambda=\inf \sigma_{ess}(J)=0$ (respectively, its upper threshold $ \lambda=\sup \sigma_{ess}(J)=4$) 
 if   the function 
\begin{equation}\label{count}
V\mapsto \tr \, E_{J+V}((-\infty,0)) \quad (\text{resp.}\quad  V\mapsto  \tr\,  E_{J+V}((4,\infty)))
\end{equation}
is discontinuous  at $V=V_0$.

Here  $E_{J+V}(\cdot)$ stands for the spectral measure associated with the self-adjoint operator $J+V$.
\end{definition}

 As Corollary  \ref{inz} suggests, the answer to   whether the operator $J+V$ is critical  at a given threshold is determined by the behavior of the perturbation  determinant  as a function of the spectral parameter $z$ in the neighborhood of the threshold. Our nearest goal is to show that  the capture of an eigenvalue by the continuous spectrum leads to the formation  of a super-resonant  (virtual) state. These states are characterized by the existence of a bounded solution to   the corresponding Schr\"odinger equation,
cf. \cite[Theorem 9.7]{Yafaev22}.

\begin{theorem}\label{mu11} 
The operator $J_n$  given by \eqref{jac}
is critical at the left  (resp. right) threshold  if and only if the threshold Jost function $C_n=j_0^{(n)}(1)$ vanishes,
that is, 
$$C_n=0
$$
(resp. $j_0^{(n)}(-1)=0$).

In this case, left (resp. right) threshold is a virtual level and  the  Jost solution 
$
\Psi^{(\ell)}=\{j_k^{(n)}(1)\}_{k=1}^\infty$  (resp. $  \Psi^{(r)}= \{j_k^{(n)}(-1)\}_{k=1}^\infty
$)
is a bounded solution   (virtual state) of the 
equation 
$$
J_n\Psi^{(\ell)}=0 \quad (\text{resp.}\,\,J_n\Psi^{(r)}=4 \Psi^{(r)}).
$$
 In particular,
 \be\label{pm}
 \Psi_k^{(\ell)}=1\quad \text{ and }\quad \Psi^{(r)}_k=(-1)^k, \quad k\ge  \max\{2,n\}.
 \end{equation}

\end{theorem}
\begin{proof} 
We discuss the case of the left threshold first.

{\it ``Only If Part".} We will prove the contrapositive statement.

Suppose $j_0^{(n)}(1)\ne0$ for  some $V$. Since  the Jost 
function $j_0^{(n)}(\Theta)=j_0^{(n)}(V;\Theta)$
 is a continuous function in $V$ and  $j_0^{(n)}(0)=1$,  in a neighborhood   of $V$
 the number of zeros of the corresponding Jost solution on $[0,1]$ remains the same, by Rouche's Theorem in the form of Hurwitz (see, e.g. \cite{Zyg}).
That is, the function $V\mapsto  \tr E_{J_n}((-0, \infty))$ is continuous in that neighborhood,
and hence, the operator $J_n$ is not critical in this case.  

{\it ``If Part".} Suppose that  
 \begin{equation}\label{nuli}
j_0^{(n)}(1)=0.
 \end{equation}
 By Proposition \ref{propjost},  we have the representation
 \be\label{sima} j_0^{(n)}(1)=\frac12(q_n-q_{n-1}),
 \ee
 where $q_n$ are  determined by the recurrent relations \eqref{recq} 
with the initial data  \eqref{indat}.

 Therefore,   \eqref{nuli} implies
 $
 q_n=q_{n-1}
 $.
From  \eqref{recq} and \eqref{sima} it also  follows that 
 \be\label{der}
 \frac{\partial}{\partial v_n}j_0^{(n)}(1)=q_{n-1}\ne 0,
 \ee
 otherwise, $q_n=q_{n-1}=0$ which in view of \eqref{recq} and \eqref{indat} would imply
 $q_0=0$.  
 Now, since \eqref{der} holds,  one can find two potentials $V_\pm$  sufficiently close (in the operator norm) to $V$ and supported on $[1,\dots, n]$
 such that  
 \be\label{ton} 
 j_0^{(n)}(V_+;1)j_0^{(n)}(V_-,1)<0.
 \ee

Therefore, the polynomial $$
 P(\Theta)=j_0^{(n)}(V_+;\Theta)j_0^{(n)}(V_-;\Theta),
\quad \Theta\in [0,1],$$ is positive at  the point $\Theta=0$  ($P(0)=1$ as it follows from  by Corollary \ref{teit}) and negative at $\Theta=1$ (by \eqref{ton}). In particular, the number of  zeros  of $P(\Theta)$ in $[0,1]$ 
  counting multiplicity is odd, which implies that the number of zeros of the factors   $ j_0^{(n)}(V_+;\Theta)$ 
  and $j_0^{(n)}(V_-;\Theta)$  of the polynomial $P(\Theta)$ in the interval $(0,1)$ are different. Hence,
$$
\tr E_{J+V_+}((-0, \infty))\ne \tr E_{
J+V_-}((-0, \infty)),
$$
for the zeros   of the Jost function of an operator in $(0,1)$ are in one-to-one correspondence with the eigenvalues of the operator on the semi-axes $(-\infty,0)$.

Since the potentials  $V_\pm$ can be chosen to be as close (in the operator norm) to $V$ as we wish, the operator $J_n(V)$ is critical in this case. 
This completes the proof of the theorem as far as the left threshold $z=0$ is concerned.

The proof of the  remaining statement 
regarding the right threshold $z=4$  is analogous.

The last assertion of the theorem  is a corollary of  the definition of the  Jost function and Jost solution
(see Section 2).

Indeed,  since the difference $J_n-J$ is a finite rank  potential supported on the integer  interval $[1, \dots, n]$, 
for the Jost function we have the representation 
$$
j^{(n)}_k(\Theta)=\Theta^k, \quad k\ge \max\{2,n\}, \quad \Theta\in \bbC.
$$
If the operator $J_n$ is critical at the left threshold $\lambda=0$ and therefore $C_n=j^{(n)}_0(1)=0$, from \eqref{JJ} it follows that 
the (bounded) sequence
 $$
\Psi^{(\ell)}_k=j^{(n)}_k(1), \quad k\ge 1, 
$$
solves the equation 
$
J_n\Psi^{(\ell)}=0 $
and \eqref{pm} holds.

The case of the right threshold is treated in an analogous manner.

\end{proof}

\section{Nodal hypersurfaces of the polynomials $Q_n$  and the Jost function $C_n$}
 
 From now on, we adopt a new notation: rather than representing the potential $V$ as a vector of strengths $(v_1, v_2, \dots)$, we will represent  it  as  
  $$V=\text{diag}[{2\mu_1},\mu_2,\dots],
 $$
 with 
 $$v_1=2\mu_1,\quad v_k=\mu_k,\quad k\ge 2.
 $$
 We change the notation here for historical reasons: in  
 \cite{LO'zdemir:2016,
LACAOT:2023,LakMot}
the interaction potential $V$ was represented in this form, and we adopt the same convention to facilitate comparison with those results.
 
 Introduce the polynomials
 $$
 Q_n=\frac12q_n|_{z=0},
 $$
 where $q_n$ are the polynomials referred to in Proposition \ref{propjost}.
 
It follows that  the Jost function  $C_n=j_0^{(n)}(1)$ on the left threshold can be represented as
 \be\label{CDD}
 C_n=Q_{n}-Q_{n-1}.
 \end{equation}

From Corollary \ref{teit} it follows that the polynomials $Q_n(\mu_1, \dots, \mu_n)$ satisfy the recurrence relations
\begin{equation}\label{Qrelations}
 Q_{n+1}=(2+\mu_{n+1})Q_n-Q_{n-1}, \quad n\ge 1,
 \end{equation}
 \begin{equation}\label{Qrelations1}
Q_0=1,
 \ee
 $$Q_1(\mu_1)=1+\mu_1.
 $$
 The key observation for our further considerations is 
  that
 the affine variety  ${\bf V}(Q_{n+1})$,  associated with the polynomial $Q_{n+1}$
is  the graph of 
  the  rational function
\begin{equation}\label{defphi}
\Phi_n(\mu_1,\dots, \mu_n )=\frac{Q_{n-1}(\mu_1,\dots, \mu_{n-1})}{Q_n(\mu_1,\dots,  \mu_{n})}-2
\end{equation} defined    on  the  domain 
\begin{equation}\label{morse}
\Dom(\Phi_n)=\bbR^n\setminus {\bf V}(Q_n).
\end{equation}

 It follows from the recurrence relations \eqref{Qrelations} and \eqref{Qrelations1} that their solutions cannot have two consecutive vanishing terms. Therefore, the polynomials $Q_{n-1}$ and $Q_n$ do not vanish simultaneously, and consequently no cancellation occurs in \eqref{defphi}. Hence, $\Dom(\Phi_n)$ is the maximal domain on which $\Phi_n$ is well-defined.

It also follows from  \eqref{Qrelations} and \eqref{defphi} that 
\begin{equation}\label{terr}{\bf V}(Q_{n+1})=\graph(\Phi_n),
\end{equation}
and hence
\be\label{domgr} \Dom (\Phi_{n+1})=(\graph(\Phi_n))^c.
\ee
 
Combining   \eqref{terr} and \eqref{domgr} shows that the sequence of functions  $\Phi_n$ can be  defined recursively: the graph of $\Phi_n$ determines the structure of the singular set of the function  $\Phi_{n+1}$ in the next generation. In this way, the family may be regarded as a  ``dynamical system'' acting on singularities, with the effective parameter space growing in dimension as $n$  
 increase:
$$
\dots \to \mathbf{V}(Q_n) \to \Dom(\Phi_n) \to \graph(\Phi_n) \to \mathbf{V}(Q_{n+1}) \to \dots. 
$$

The key geometric observation is this setting is the following   result.

\begin{theorem}\label{decomm}
 The domain of the rational function $\Phi_n$  decomposes into  $n+1$  unbounded connected  components,
  \begin{equation}\label{dom}
\Dom (\Phi_n)=\bigcup_{k=0}^n D_k^{(n)}.
\end{equation}

Here
$$D_1^{(1)}=(-\infty,-1)\quad \text{ and} \quad  D_0^{(1)}=(-1,\infty)$$
and the components in higher dimensions are defined inductively:

\begin{itemize}
 \item[]$D_0^{(n+1)}$ is the ephigraph of $\Phi_n$ restricted to $D_0^{(n)}$;
 \item[] $D_k^{(n+1)}$, $k=1, \dots n$  is the hypograph of $\Phi_n$ restricted to $D_{k-1}^{(n)}$
 joined together with the epigraph of $\Phi_n$ restricted to $D_k^{(n)}$ along their shared boundary;
\item[] and
\item[]$D_{n+1}^{(n+1)}$ is the hypograph of $\Phi_n$ restricted to $D_n^{(n)}$.

\end{itemize}
 \end{theorem}
 
\begin{proof}
The claim  can easily be  verified in low dimensions \cite{LO'zdemir:2016,LACAOT:2023}.

 Indeed, if  $n=1$,  we have 
$$ Q_1(\mu_1)=1+\mu_1, 
$$
and therefore  ${\bf V}(Q_1)$ is a one-point set,  ${\bf V}(Q_1)=\{-1\}$.
Hence, the complement $({\bf V}(Q_1))^c=\bbR\setminus \{-1\}$ has two connected components,
$$D_1^{(1)}=(-\infty,-1)\quad \text{ and} \quad  D_0^{(1)}=(-1,\infty),$$
so that
$$
\Dom (\Phi_1)=({\bf V}(Q_1))^c=\bigcup_{k=0}^1 D_k^{(1)}.
$$

To see the pattern,   suppose next that  $n=2$. For the polynomial $Q_2(\mu_1,\mu_2)$ we have the representation
\begin{align*}
Q_2(\mu_1,\mu_2)=1+2\mu_1+\mu_2+\mu_1\mu_2.\end{align*}
Therefore,  ${\bf V}(Q_2)$ consists of two branches of the hyperbola
$$
1+2\mu_1+\mu_2+\mu_1\mu_2=0.
$$
In particular, 
\begin{equation}\label{n=2}
\Dom (\Phi_2)=({\bf V}(Q_2))^c=\bigcup_{k=0}^2 D_k^{(2)},
\end{equation}
where $D_2^{(2)}$ in the part of the plane $\bbR^2$  lying below  the branch of the hyperbola  in the left half-plane, 
$D_1^{(2)}$ lies between the two branches of the hyperbola,  and  $D_0^{(2)}$ is the part of the plane  located above the second branch of the hyperbola.

Now, we can proceed by induction.  

Suppose that  the partition \eqref{dom} has been already established for some  $n\ge 2$.

Observe that 
 the graph   of the rational  function $\Phi_n$ splits the cylinders $D_k^{(n)}\times\R,\, k=0,...,n$ constructed over the connected components $D_k^{(n)}$ of the domain 
of $\Phi_n$ into upper and lower parts, the epigraphs and hypographs, respectively:
\begin{equation}\label{nott}
    U^{(n)}_k=\text{epi}(\Phi_n|_{D^{(n)}_k})
 \quad \text{and}\quad 
  L^{(n)}_k=\text{hypo}(\Phi_n|_{D^{(n)}_k}),
  \end{equation}
  separated  by the graph of the restriction of the function  $\Phi_n$ onto the component $D^{(n)}_k$
\begin{equation}\label{gammy}
  \Gamma^{(n)}_k=\graph(\Phi_n|_{D^{(n)}_k}),\quad k=0,1,\dots,n.
  \end{equation}
 That is, 
  $$
D_k^{(n)} \times \bbR= L_k^{(n)}\cup \Gamma_k^{(n)}
\cup U_k^{(n)}.
$$

Here  we have used the following notation  for the  (open)   epigraph and hypograph of a function $f$:

$$\text{epi}(f)=\{(x,y)\in \Dom(f)\times \bbR\,|\,y>f(x)\}$$ 
and  
$$\text{hypo}(f )=\{(x,y)\in \Dom(f)\times \bbR\,|\,y<f(x)\},$$ respectively.

The upper and lower parts \eqref{nott}  are then joined together along their shared boundary, forming the connected components of the domain for the  function $\Phi_{n+1}$ as follows:
define the sets  $D^{(n+1)}_k,\quad k=0,\dots, n+1$ from the next generation: for $k=0$  and $k=n+1$   we set 
\begin{equation}\label{iind}
D_0^{(n+1)} = U_0^{(n)}, \qquad D_{n+1}^{(n+1)} = L_n^{(n)},
\end{equation}
and  for $ k = 1, \dots, n $ we define
\begin{equation}\label{iind1}
D_k^{(n+1)} = L_{k-1}^{(n)} \cup \left( \Gamma_{k-1}^{(n-1)} \times \mathbb{R} \right) \cup U_k^{(n)}.
\end{equation}

We need  to show  that decomposition of the definition domain into components \eqref{dom} holds in dimension  $ n+1$, that is, 
\begin{equation}\label{kokoko}
\Dom (\Phi_{n+1})=\bigcup_{k=0}^{n+1} D_k^{(n+1)}.
\end{equation}

Indeed, from \eqref{morse} it follows the space partition
\begin{align*}
\bbR^{n+1}&=(\Dom(\Phi_n)\times \bbR)\cup (\text{Gr}(\Phi_{n-1})\times\bbR)
\\&=\graph(\Phi_n)\cup \bigcup_{k=0}^n \left (L_k^{(n)}\cup U_k^{(n)}\right )\cup \bigcup_{k'=0}^{n-1}
\left (\Gamma_{k'}^{(n-1)}\times\bbR\right ).
\end{align*}
In view of \eqref{domgr} we get
\begin{equation}\label{ogo}
\Dom (\Phi_{n+1})=\bigcup_{k=0}^n \left (L_k^{(n)}\cup U_k^{(n)}\right )\cup \bigcup_{k'=0}^{n-1}
\left (\Gamma_{k'}^{(n-1)}\times\bbR\right ),
\end{equation}
which yields \eqref{kokoko} (by reshuffling the sets from the right hand side of \eqref{ogo} and using \eqref{iind} and \eqref{iind1}).

 By induction, one also
 shows  that the set $ \Gamma_{k-1}^{(n-1)} \times \mathbb{R} $ from \eqref{iind1} is the joint boundary of
  $L_{k-1}^{(n)}$ and $   U_k^{(n)}$,
  which ensures that the components $D^{(n+1)}_k$, $k=0,\dots, n+1$ are connected.
  
 \end{proof}
The affine variety ${\bf V}(Q_n)$, the zero set of the polynomial $Q_n$, and its complement admit the following description.
\begin{corollary}\label{nuliQ}
The affine variety ${\bf V}(Q_n)$ 
 is a disjoint union of $n$ smooth $(n-1)$-dimensional surfaces
 $$
 {\bf V}(Q_n)=\bigcup_{k=0}^{n-1} \Gamma_k^{(n-1)},
 $$
 where
 $$
\Gamma_k^{(n-1)}=\graph (\Phi_{n-1}|_{D^{(n-1)}_k}),\quad k=0,\dots, n-1,
$$
is   the graph of the  rational function  $\Phi_{n-1}$, restricted to the connected component  $D^{(n-1)}_k$ of its domain,  as described in Theorem \eqref{decomm}.

 Moreover, the complement  $({\bf V}(Q_n))^c$ of the affine variety  ${\bf V}(Q_n)$ is a disjoint union of 
$n+1$ unbounded open connected components corresponding to  the domain decomposition of the rational function of the next generation $\Phi_n$ (see  \eqref{dom}). That is, 
$$
({\bf V}(Q_n))^c=\bigcup_{k=0}^n D_k^{(n)}.
$$

\end{corollary}

 Now we can show  that the zero-level set of the Jost threshold function $C_n$ is   the graph of a rational function consisting of $n$  smooth, unbounded surfaces that partition the space into $n+1$ connected components.

\begin{theorem}\label{main00}  The affine varieties ${\bf V}(C_n)$ and  ${\bf V}(Q_{n})$
are related as 
\be\label{shity}
{\bf V}(C_{n})=(\underbrace{0,\dots, 0}_{(n-1)\text{ times}},1)+{\bf V}(Q_{n}).
\ee

In particular, the affine variety ${\bf V}(C_n)$
is the  disjoint union  of $n$ smooth surfaces $$
{\bf V}(C_n)=\bigcup_{k=0}^{n-1} \graph\left((\Phi_{n-1}+1)|_{D^{(n-1)}_k}\right ),
$$
and its compliment is a disjoint union 
$n+1$ unbounded path-connected open components,
\begin{equation}\label{split}
({\bf V}(C_n))^c=\bigcup_{k=0}^n G_k^{(n)},
\end{equation}
where 
\begin{equation}\label{concom}
 G_k^{(n)}=(\underbrace{0,\dots, 0}_{(n-1)\text{ times}},1)+D_k^{(n)},\quad k=0,1,\dots, n, 
\end{equation}
and 
$ D_k^{(n)}$, $k=0,1, \dots, n$ are the sets referred to in Theorem \ref{decomm}. 

\end{theorem}
\begin{proof}

Using the relation  
${\bf V}(Q_{n})=\graph(\Phi_{n-1})
$ (see eq. \eqref{terr} and \eqref{CDD}),
we see  that 
\begin{equation}\label{terrr}{\bf V}(C_{n})=\graph(\Phi_{n-1}+1),
\end{equation}
where $\Phi_n$ is a rational function given by \eqref{defphi}.
Therefore, \eqref{shity} holds 
and then \eqref{split} and  \eqref{concom} follow from  Theorem \ref{decomm} and Corollary \ref{nuliQ}.

\end{proof}

 In conclusion, we present several results that shed additional light on the geometry of the connected components  $D_k^{(n)}$ from the domain decomposition \eqref{dom} at infinity and will be needed in what follows.
 
 \begin{lemma}\label{kis} The set $D_0^{(n)}$ referred to in Theorem \ref{decomm} contains the origin.
 \end{lemma}
 \begin{proof} 
 Induction on $n$.

 The base of the induction, $n=1$: 
 $$0\in(-1,\infty)=D_0^{(1)}.$$
 
 Now, suppose that 
 $$
 \bbR^n\ni(0, \dots, 0)\in D_0^{(n)}$$
 for some space dimension $n$.  
 
From the definition  of the rational function $\Phi_n$ (see \eqref{defphi})  it follows that 
$$\Phi_n(0, \dots, 0)=-1
$$
and hence 
$$
 \bbR^{n+1}\ni(0,\dots, 0,-1)\in \graph(\Phi_n|_{D^{(n)}_0}).
$$
Therefore, the origin of the space $\bbR^{n+1}$
belongs to  the ephigraph of $ \Phi_n|_{D^{(n)}_0}$,  which,  by Theorem \ref{decomm},  coincides with 
the set  $ D_0^{(n+1)}$. Therefore,
\begin{equation}\label{prik}
 \bbR^{n+1}\ni(0,\dots, 0)\in 
D_0^{(n+1)}.
 \end{equation}
 This concludes the proof of the inductive step and hence the theorem.

 \end{proof}
 
 \begin{lemma}\label{cone} Given $n\in \bbN$, for $k=1,\dots, n$,
denote by $e_k^{(n)}$ the $n$-dimensional vector
$$
e_k^{(n)}=(\underbrace{0, \dots, 0}_{n-k\text{ times}}, \underbrace{-1, \dots, -1}_{k\text{ times}}).
$$

Then 
\begin{equation}\label{mir}
te_k^{(n)}\in D_k^{(n)}, \quad \text{for all} \quad  t>0 \quad \text{large enough},
\end{equation}
where   $D_k^{(n)}$ are the connected components referred to in Theorem \ref{decomm}.
\end{lemma}
\begin{proof} Induction on $n$.

The base of the induction, $n=1$.
We have
$$
te_1^{(1)}=-t, \quad t>0.
$$
Therefore, 
$$ te_1^{(1)}\subset (-\infty,-1)=D_1^{(1)} \quad t\ge 1,$$
 which justifies \eqref{mir} in the one-dimensional case $n=1$.

Now, suppose that \eqref{mir} holds for any $k=1, \dots, n$  for some space dimension $n$.  
Equivalently, 
$$
te_{k-1}^{(n)}\in D_{k-1}^{(n)}, \quad k=2, \dots, n+1,\quad t>0\quad  \text{large enough}.
$$

It follows,  
 \begin{equation}\label{milo}
(\underbrace{0,\dots,0}_{n-(k-1)\text{ times}}, \underbrace{-t, \dots,-t}_{(k-1)\text{ times}}, \Phi_n(te_{k-1}^{(n)}))\in 
\graph(\Phi_n|_{D^{(n)}_{k-1}}),
\end{equation}
where $\Phi_n$ is the rational function given by \eqref{defphi}.
Since 
 $$\lim_{t\to \infty}\Phi_n(te_{k-1}^{(n)}))=-2,
 $$
from \eqref{milo} we see that  possibly for a larger $t$ the point $te^{(n+1)}_k$
 belongs to the hypograph of  
 $\Phi_n|_{D^{(n)}_{k-1}}$, which, by Theorem \ref{decomm}, is a subset of $ D_k^{(n+1)}$. That is,
 $${\text{hypo}}(\Phi_n|_{D^{(n)}_{k-1}})\subset D_k^{(n+1)} ,$$ and hence  membership \eqref{mir} 
 takes place in the space dimension $n+1$  for all $2\le k\le n+1$.   
 
 In the remaining case, $k=1$, we argue as follows. 

By Lemma \ref{kis},  the origin is a point of $ D^{(n)}_{0}$.  Therefore,  
$$(\underbrace{0,\dots, 0}_{n},\Phi_n(0,\dots, 0))=e_1^{(n+1)}\in\graph(\Phi_n|_{D^{(n)}_{0}}).$$
Therefore, 
$$
te_1^{(n+1)}\in{\text{hypo}}(\Phi_n|_{D^{(n)}_{0}})\subset D_1^{(n+1)}, \quad t>1,
$$
which proves \eqref{mir} for $k=1$  in  the space dimension $n+1$.
\end{proof}

\section{A simple perturbation result}
 In what follows, we are interested in describing the location of the discrete spectrum for finite-dimensional perturbations of a Jacobi matrix in the large-coupling-constant limit. In this regime, the free Jacobi operator \(J\) is naturally regarded as a perturbation of a finite-dimensional potential \(V\). For our purposes, it is sufficient to consider interaction potentials of a special form. We therefore assume that \(V\) is a finite-rank, self-adjoint partial isometry, that is, the difference of two orthogonal projections.

\begin{lemma}\label{nulli}
Let $\cI$ and $\cJ$ be $k$- and $\ell$-element  disjoint   subsets  of   $\bbN$.  Denote by $P_k$ and $Q_\ell$ the orthogonal projections onto the subspaces 
$\text{span}_{i\in \cI}\{\delta_i\}$ and $\text{span}_{j \in \cJ}\{\delta_j\}$,
 respectively, with the agreement that $P_0=Q_0=0$.  Set 
 $$
 V=Q_\ell-P_k.
 $$
  
  Then for all    $t>8$ operator $J+tV$ has exactly $k$ eigenvalues below  and exactly $\ell$ eigenvalues above its continuous spectrum  $[0, 4]$.

\end{lemma}
\begin{proof}
First assume that $k, \ell>0$.

 Clearly, $\text{rank }(P_k)=k$,  $\text{rank }(Q_\ell)=\ell$ and $\text{rank }(V)=k+\ell$.
The spectrum   of the diagonal operator $tV$, $t>0$ is a three point set  $\{-t, 0, t\}$ with $-t$ an eigenvalue of multiplicity $k$, 
$t$ is an eigenvalue of multiplicity $\ell$ and 
zero  an eigenvalue of infinite multiplicity.
  Therefore, we have the splitting 
  $$\text{spec}(tV)=\sigma\cup\Sigma,
 $$
 where
 $$
 \sigma=\{-t, t\}, \quad \text{and} \quad  \Sigma=\{0\}.
 $$

 From   Proposition \ref{pams} it follows that if
 \begin{equation}\label{test}
\frac12\dist(\sigma,\Sigma)=\frac{t}{2}>\|J\|=4,
 \end{equation}
 which obviously holds true for  $t>8$,  for the difference of the corresponding spectral projections  we have the 
 norm estimate
 \begin{equation}\label{proj}
 \|E_{tV}(\delta)-E_{J+tV}(\delta )\|<1.
 \end{equation}
Here 
 \begin{equation}\label{dell}
 \delta=\left (\left (-\frac32t,- \frac12t\right )\cup\left (\frac12t, \frac32t\right )\right )\subset \bbR\setminus[0,4].
 \end{equation}
 From \eqref{proj} it follows that 
 $$
 k+\ell=\text{rank} (E_{tV}(\delta))=\text{rank}(E_{J+tV}(\delta )).
 $$
Therefore, in view of \eqref{dell},  the operator $J+tV$ has at least $k+\ell$ simple eigenvalues outside the interval $[0,4]$.
Since the total number of eigenvalues  outside  $[0,4]$ may not exceed the rank $k+\ell $ of the perturbation $tV$, we conclude that
the operator $J+tV$ has exactly  $k+\ell$ simple eigenvalues outside its continuous spectrim  $[0,4]$. 

To see that $J+tV$ has exactly $k$ negative
 eigenvalues, we argue as follows: since the operator inequality $J+tV\ge J-tP_k$ holds, we have
 $$
 \tr E_{J+tV}((-\infty, 0))\le  \tr E_{J-tP_k}((-\infty, 0))\le \text{rank }(P_k)=k.
 $$
 Analogously,  the operator inequality   $J+tV\le J+tQ_\ell$ implies
 $$
 \tr E_{J+tV}((4, \infty )\le   \text{rank }(Q_\ell)=\ell.
 $$
 However, we have shown that 
 $$
 \tr E_{J+tV}((-\infty, 0))+ \tr E_{J+tV}((4,\infty ))=k+\ell.
 $$
 Hence
 $$
 \tr E_{J+tV}((-\infty, 0))=k\quad \text{and}\quad  \tr E_{J+tV}((4,\infty ))=\ell.
 $$
 
 If $k$  vanishes and $\ell>0$ (or $\ell$  vanishes but $k>0$)  the proof   proceeds along  the same lines with the only difference being  that the set $\sigma=\{-t,0\}$   ($\sigma=\{0,t\}$) is now a two point set.

  Finally, if $k=\ell=0$, there is nothing to prove: the spectrum of the operator $J$ is absolutely continuous, there are no eigenvalues at all. 
  
 \end{proof}

\section{ The geometry of the critical variety and  the discrete spectrum }

Now we are ready to give a complete characterization of the operators $ J+V$ with 
a finite-rank potential V supported on $[1,\dots,n]$ that have a prescribed number of eigenvalues  to the right and to the left of its continuous spectrum.

We will adopt the following notation:  $$H_\mu=J+V_\mu,\quad \mu=(\mu_1, \dots, \mu_n)\in \bbR^n,$$
 where
\begin{equation}\label{tutit}
V_\mu=\text{diag}(2\mu_1,  \mu_2, \dots ,\mu_n,0,\dots). 
\end{equation}
We first consider the case where the continuous spectrum thresholds of the operator
$H_\mu$ are free of virtual levels.

\begin{theorem}\label{koko}
Assume  that  the operator
$H_\mu$
is noncritical.

Then 
\begin{itemize}
\item[(i)] $H_\mu$ has $k$  (simple) negative eigenvalues  for $ 0\le k\le  n$ if and only if
$$\mu\in  G
_k^{(n)};$$

\item[(ii)] $H_\mu$ has $\ell$
 (simple)  positive eigenvalues on the semi-axis $(4,\infty)$   for $ 0\le \ell\le n$ 
 if and only if
 $$-\mu\in  G
_{\ell}^{(n)};
$$
 \end{itemize}
 
 In particular, if   
 $H_\mu$ is non-critical and  has $k$ (simple) eigenvalues  below the threshold and $\ell$ simple eigenvalues above the threshold for some $k$ and $\ell$ such that $$0\le k+\ell\le n$$ 
  if and only if
 $$\mu\in G
_k^{(n)}\cap (-G
_\ell^{(n)}).$$

\end{theorem}

\begin{proof} 
Notice   that if $\mu,\nu\in  G_k^{(n)}$, then  the number of eigenvalues of  the operators  $H_\mu$ and $H_\nu$ below the threshold $\lambda=0$  remains the same. Indeed, since  $G_k^{(n)}$ is a connected set, one can find a path $[0,1]\ni t \mapsto \mu_t$  with 
 the end points at $\mu$ and $\nu$ such that $C_n(\mu_t)\ne 0$ along the path. Therefore,   the operators $H_{\mu_t}$ are not critical along the path. Since the pass is a compact, the number of negative eigenvalues $n_-(H_\mu)$ (counting multiplicity) of $ H_\mu$ remains constant along the path. Therefore,
 $$n_-(H_{\mu})=n_-(H_{\nu}).$$

(i). First, we show that for  $\mu\in  G
_k^{(n)},$ $ k=0,1,\dots, n$ the operator $H_\mu$ has $k$ simple negative eigenvalues.

Assume  that $k=0$.

 By Lemma \ref{kis}, $
\bbR^n\ni(0, \dots,0)\in  D_0^{(n)},
$
which implies the membership
$$
\mu_0=(0, \dots,0,1)\in  G_0^{(n)}.
$$
In particular, $H_{\mu_0}$ being a rank one non-negative perturbation of  
the free operator $J$, has no negative eigenvalues.
 Therefore, as explained above, 
 for every  $\mu$ from  the connected component $G_0^{(n)}$ the discrete spectrum of $H_\mu$ on $(-\infty, 0)$ is empty.

 Now assume that $1\le k\le n$.

 In view of \eqref{concom}, from Lemma  \ref{cone} it follows that the point $$\mu_t=(\underbrace{0,\dots, 0}_{n-k\text{ times}}, -t, \dots, -t)$$ belongs to the connected component 
 $G_k^{(n)}$  for  $t $ large enough. Without loss we may assume that $t>8$. In this case,  from  Lemma \ref{nulli} it follows that the operator  $H_{\mu_t}$ has exactly $k$ negative eigenvalues. Therefore, for every $\mu$ from the connected component $G_k^{(n)}$
 the operator $H_{\mu}$ has exactly $k$ negative eigenvalues.

Conversely, suppose that  $H_{\mu}$ has exactly $k$ negative eigenvalues. By the hypothesis, 
  $H_{\mu}$ is not critical. From  Theorem \ref{mu11}, the Jost function $C_n$ does bot vanish, and hence, $\mu\in G_m^{(n)}$ for some $m=0, 1, \dots, n$ as we see it  from \eqref{split} (Theorem \ref{main00}). Clearly, $m=k$,
  completing the proof of  (i).
  
 (ii).
 The second assertion of the theorem then follows from (i) by applying the Jost function Symmetry Principle \eqref{symm}: the operator
 $H_\mu=J+V_\mu$ has exactly $\ell$ eigenvalues below   the lower  threshold $\lambda=0$ if and only iff
  the operator $H_\mu=J-V_\mu$
has exactly $\ell$ eigenvalues above  the upper threshold $\lambda=4$.

\end{proof}
                           
 It remains to discuss the case of a critical  operator $H_\mu$  is critical, i.e., when a virtual level occurs at one or both thresholds.

\begin{theorem}\label{lala}
Suppose that the operator $H_\mu=J+V_\mu$   has a virtual level at its left threshold. Then there is a unique $k$, $0\le k\le n-1$ such that 
\be\label{mama}
\mu\in\graph\left ((\Phi_{n-1}+1)|_{D^{(n-1)}_k}\right ),
\ee
 where $\Phi_{n-1}$ is the rational function from \eqref{defphi} and 
 $D^{(n-1)}_k$
 is a connected component referred to in Theorem \ref{decomm}.

In this case, the operator $H_\mu$ has a threshold resonance  at $\lambda=0$ and  exactly $k$ simple eigenvalues on the negative real axis.

\end{theorem}
\begin{proof} 
Since $H_\mu$ is critical by hypothesis, Theorem~\ref{mu11} implies that
$\mu \in {\bf V}(C_n)$. By Theorem \ref{main00},
$$
{\bf V}(C_n)=\bigcup_{k=0}^{n-1} \graph\left ((\Phi_{n-1}+1)|_{D^{(n-1)}_k}\right ),
$$
and therefore  \eqref{mama} holds for some $k$, $0\le k\le n-1$. That is, 
$$
(\mu_1, \dots, \mu_{n-1})\in D^{(n-1)}_k
$$
and 
\be\label{epii}\mu_n=\Phi_{n-1}(\mu_1, \dots, \mu_{n-1})+1.\ee

 Using  \eqref{epii}, we see that  for any  $\varepsilon >0$,
 the point $$
\mu_\varepsilon=(\mu_1, \dots, \mu_{n-1}, \mu_n+\varepsilon)
$$
 belongs to the ephigraph   of  the function  $\Phi_{n-1}+1$ restricted to $D^{(n-1)}_k$.
 From Theorem \ref{decomm} it follows that   the ephigraph of 
  $(\Phi_{n-1}+1)|_{ D^{(n-1)}_k}$ is a subset of  the connected  component $D_k^{(n)}$ shifted by the vector 
  $(\underbrace{0,\dots, 0}_{(n-1)\text{ times}},1)$. Taking into account  \eqref{concom},
 we conclude that      for any $\varepsilon>0$ we have the membership
 $\mu_\varepsilon\in  G^{(n)}_k$.
Therefore, by Theorem~\ref{koko}, the operator $H_{\mu_\varepsilon}=J+V_{\mu_\varepsilon}$ has exactly $k$ negative eigenvalues; that is,
$$
n_-(H_{\mu_\varepsilon})=k.
$$ On the other hand,  $H_{\mu_\varepsilon}$
 is a positive (rank one) perturbation of  $H_{\mu}$, and  hence, 
the number of  negative eigenvalues of   $H_{\mu_\varepsilon}$ can at most  decrease.
That is,
$$
n_-(H_{\mu})\ge n_-(H_{\mu_\varepsilon})=k.
$$
Next, choosing  $\varepsilon$  
smaller than the distance  of the nearest negative eigenvalue  of $H_{\mu}$  to the threshold at $\lambda=0$,
we also see that the reverse inequality
$$
n_-(H_{\mu})\le n_-(H_{\mu_\varepsilon})=k
$$
holds.

 Therefore,
$$n_-(H_{\mu})=k,
$$
which completes the proof.

\end{proof}

\begin{remark}\label{posle}
As far as the left threshold $\lambda=0$ is concerned, the results of  Theorem  \ref{koko} (i) and  Theorem \ref{lala} can be summarized as follows. 
\begin{itemize}
\item[]The zero set of the polynomial $C_n$ decomposes the parameter space of the interaction potentials  into $n$ smooth hypersurfaces. On each hypersurface, the operator becomes critical: one of its negative eigenvalues reaches the lower threshold, leading to the formation of a virtual state.
When the parameter crosses such a hypersurface, the spectral characteristics of the operator change discretely: the number of negative eigenvalues varies by exactly one. 
\end{itemize}
More precisely, each critical hypersurface is the graph of a function, which naturally determines its orientation and distinguishes the lower and upper sides. Passing the parameter from the lower to the upper side corresponds to a negative eigenvalue being absorbed by the continuous spectrum.
\end{remark}

 An analogous description applies to the right threshold $\lambda=4$.

\begin{theorem}\label{rara}
Suppose that   $H_\mu=J+V_\mu$  has a virtual level  at its right  threshold, then   there  is a unique $\ell$, $0\le \ell\le n-1$ such that 
$$
-\mu\in\graph\left ((\Phi_{n-1}+1)|_{D^{(n-1)}_\ell}\right ).
$$
In this case, the operator $H_\mu$  has  a threshold resonance  at $\lambda=4$ and  exactly $\ell$ simple eigenvalues on the semi-axis $(4,\infty)$.
\end{theorem}
\begin{remark}
Just as in Remark \ref{posle}, one can formally describe the creation and annihilation of eigenvalues associated  at the right threshold. Moreover, the description remains valid in the presence of a virtual level at both  the right and  left thresholds. The only difference is that a more detailed analysis  is required of the intersection geometry
of the affine variety ${\bf V}(C_nC_n^*)$ (see the Jost Function Symmetry/Reflection Principle \eqref{symm}),  where the polynomial $C_n^*$ is defined as
$$
C_n^*(\mu)=C_n(-\mu), \quad \mu\in \bbR^n.
$$
\end{remark} 

In conclusion, notice that  from an analytic point of view, the formation of a virtual level at a threshold, as the perturbation parameters vary, is related to the collision of a zero of the determinant,  being a rational function of the local parameter $\Theta$,  with its pole at the threshold, which ultimately removes the threshold singularity of the perturbation determinant.
 Notice that if virtual levels occur at {\sc both edges} of the continuous spectrum, then the perturbation determinant has removable singularities at both threshold points and therefore the perturbation determinant becomes a polynomial in $\Theta$.

\section*{Acknowledgement}
This work has been started when the first author (SL) visited
the University of Missouri, Columbia, as a Miller scholar in  November and December
of 2024. He is grateful to the Department of Mathematics for its hospitality. SL also acknowledge support of this research by Ministry of Higher Education, Science and Innovation of the Republic of Uzbekistan (Grant No. FL-9524115052). 
We gratefully acknowledge M. O. Akhmadova and S. I. Fedorov for valuable discussions.
\appendix

\section{The Jost function and  perturbabtion determinant. Proofs}
\subsection{Proof of Proposition \ref{propjost}}\label{A1}
\begin{proof} Assume first that $n\ge 3$. From the definition of the  Jost solution $\{j_k^{(n)}\}_{k=0}^\infty$ it follows 
that its first $n+1$ components  $j_0^{(n)}, j_1^{(n)}, \dots j_n^{(n)}$ satisfy the system of equations 
 \begin{equation}\label{matsys}\begin{pmatrix}
-\sqrt{2}&m_1&-\sqrt{2}&0&\dots&0\\
\dots&-\sqrt{2}&m_2&-1&\dots&\dots
\\ 0&\dots&-1&m_3&-1&\dots
\\
\dots&\dots&\dots&\dots&\dots&\dots\\
0&\dots&\dots&\dots&\dots&-1
\\0&\dots&&&-1&m_{n}
\\0&\dots&&&0&1
\end{pmatrix}
\begin{pmatrix}
j_0^{(n)}\\
j_1^{(n)}\\
\dots\\
\dots\\
\dots\\
j_{n-1}^{(n)}\\
j_{n}^{(n)}
\end{pmatrix}=\begin{pmatrix}
0\\
0\\
\dots\\
\dots\\
0\\
\Theta^{n+1}(z)
\\
\Theta^n(z) 
\end{pmatrix},
\end{equation}
where
\be\label{cfcf}
m_k=2-z+v_k.
\ee
Let $M_{ij}$ denote the minor of the entry in the $i$-th row and $j$-th column of the matrix of the system \eqref{matsys}.
Since the first row of the inverse matrix of the system  consists of the cofactors of the first column 
of the system matrix divided by its determinant, we explicitly get 
\begin{align*}
j_0^{(n)}&=
\frac{1}{2(-1)^n}(\Theta^{n+1}(z)M_{n1}(-1)^{n+1} +\Theta^n(z)M_{(n+1)1}(-1)^{n+2}))
\\&=\frac{ \Theta(z)^n}{2}(M_{(n+1)1}-\Theta(z)M_{n1}).
\end{align*}

Clearly, 
$$
M_{(k+1)1}=q_k,\quad k=n, n+1, 
$$
where
$q_k$ are the   polynomials (in $z$ and $v_1, \dots, v_k$) given by 
\begin{equation}\label{poli}
q_k=\begin{vmatrix}
2-z+v_1&-\sqrt{2}&0&\dots&0\\
-\sqrt{2}&2-z+v_2&-1&\dots&\dots
\\0&-1&2-z+v_3&-1&\dots
\\
\dots&\dots&\dots&\dots&\dots&\\
\dots&\dots&\dots&\dots&-1&
\\0&\dots&\dots&-1&2-z+v_k
\end{vmatrix}.
\end{equation}
Thus, for the Jost function  $j_0^{(n)}$ we obtain  the rerpesentation
$$
j_0^{(n)}=\frac{ \Theta(z)^n}{2}(q_n-\Theta(z)q_{n-1}), \quad n\ge 3.
$$

Expanding the determinant \eqref{poli} by minors on the last row yields the recurrence
\be\label{simm}
q_{k+1}=(2-z+v_{k+1})q_{k}-q_{k-1},\quad k\ge 2.
\ee

If $n=2$, for determining the Jost function  $j_{0}^{(2)}$ we have the system of equations
\be\label{matn=2}\begin{pmatrix}
-\sqrt{2}&m_1&-\sqrt{2}\\
0&-\sqrt{2}&m_2
\\ 
0&0&1
\end{pmatrix}
\begin{pmatrix}
j_0^{(2)}\\
j_1^{(2)}\\
j_{2}^{(2)}\\
\end{pmatrix}=\begin{pmatrix}
0\\
\Theta^{3}(z)
\\
\Theta^2(z) 
\end{pmatrix}.
\ee
Solving \eqref{matn=2} for $j_0^{(2)}$
yields
$$
j_0^{(2)}=
\frac{ \Theta(z)^2}{2}(m_1m_2-2-\Theta(z)m_1).
$$
Therefore,
$$
j_0^{(2)}=\frac{ \Theta(z)^2}{2}(q_2-\Theta(z)q_1),
$$
where 
$$
q_1=m_1=2-z+v_1
$$
and 
$$q_2=m_1m_2-2= (2-z+v_2)q_1-q_0
$$
with 
$$
q_0=2,
$$
which also shows that \eqref{simm} holds for $k=1$.

The obtained representations prove \eqref{jn} for $n\ge 2$.

One also observes that  for $n=1$ the Jost function $ j_0^{(1)}$ can be recovered from the obvious relation
$$
j_0^{(1)}=j_0^{(2)}|_{v_2=0}.
$$
Having this in mind and  using \eqref{guk},  we see that 
\begin{align*}
j_0^{(1)}&=\frac{ \Theta(z)^2}{2}((m_1(2-z)-2-\Theta(z)m_1)
\\&=\frac{ \Theta(z)^2}{2}
\left (\frac{m_1}{\Theta(z)} -2 \right )
=\frac{ \Theta(z)}{2}(m_1-2\Theta(z))
\\&=\frac{ \Theta(z)}{2}(q_1-\Theta(z)q_0),
\end{align*}
which also shows that \eqref{jn} takes place for $n=1$ as well.

For the Jost function  associated with  the unperturbed Jacobi matrix we also obtain
$$
j_0^{(0)}=j_0^{(1)}|_{v_1=0}=\frac{ \Theta(z)}{2}(2-z-2\Theta(z))=\frac{1-\Theta^2(z)}{2},
$$
which competes the proof of the proposition.

 \end{proof}

\subsection{Proof of Proposition \ref{propdet}}\label{A2}
\begin{proof} Let $\delta_1, \delta_2, \dots$ be the standard basis in the Hilbert space $\ell^2(\bbN)$.
Since the operator $J_1$ is a rank one perturbation of $J$, for the  perturbation determinant 
$  \text{det}_{J_1/J}(z)$  associated with the pair $(J, J_1)$ we have 
$$
\text{det}_{J_1/J}(z)=1+v_1((J-z)^{-1}\delta_1,\delta_1), \quad \Im(z)\ne0.
$$

The $m$-function $$f_1= ((J-z)^{-1}\delta_1,\delta_1), \quad \Im(z)\ne 0, 
$$ can easily be recovered by solving the system of equations
\begin{equation}\label{green2}\begin{pmatrix}
2-z&-\sqrt{2}
\\
-\sqrt{2}&2-z-\Theta
\end{pmatrix}
\begin{pmatrix}
f_1\\
f_2\\
\end{pmatrix}
=\begin{pmatrix}
1\\
0\\
\end{pmatrix},
\end{equation}
where $\Theta=\Theta(z)$ is a root of  \eqref{guk} such that $|\Theta(z)|<1$ for $\Im(z)\ne 0$.
Solving \eqref{green2} we obtain  
$$
f_1=\frac{2-z-\Theta}{(2-z)(2-z-\Theta)-2}=
\frac{\Theta}{1-\Theta^2}
$$
and hence
\begin{equation}\label{kkk}
\text{det}_{J_1/J}(z)=1+v_1f_1=\frac{1+v_1\Theta-\Theta^2}{1-\Theta^2}
.\end{equation}
Using \eqref{guk}, we see that 
$$
1+v_1\Theta-\Theta^2=\Theta\left (\frac1\Theta+v_1-\Theta\right )
=\Theta (2-z+v_1-2\Theta),
$$
which, due to \eqref{indat}, amounts to 
\begin{equation}\label{ww}
1+v_1\Theta-\Theta^2=\Theta (q_1-\Theta q_0).
\end{equation}

Therefore, 
$$
\text{det}_{J_1/J}(z)= \frac{\Theta(q_1-\Theta q_0)}{1-\Theta^2}=\frac{2}{1-\Theta^2}\frac{\Theta(q_1-\Theta q_0)}{2}
.
$$
By Proposition \ref{propjost},
$$
\frac{\Theta(q_1-\Theta q_0)}{2}=j_0^{(1)}(z),
$$
hence
\be\label{comb1}
\text{det}_{J_1/J}(z)=\frac{2}{1-\Theta^2}\cdot j_0^{(1)}(z)=\frac{ j_0^{(1)}(z) }{j_0^{(0)}(z)},
\ee
which  proves \eqref{vita} for $n=1$.

More generally, one shows that the first $n$ components $f_1,f_2, \dots f_n$ from the expansion 
$$
(J_{n-1}-z)^{-1}\delta_{n}=\sum_{k=1}^\infty f_k\delta_k, \quad \Im(z)\ne 0, 
$$
can be determined by  solving the system of equations

\begin{equation}\label{green}\begin{pmatrix}
m_1&-\sqrt{2}&0&\dots&\dots&0
\\
-\sqrt{2}&m_2&-1&0&\dots&\dots
\\ 
0&-1&m_3&-1&\dots&\dots
\\
\dots&\dots&\dots&\dots&\dots&\dots
\\
\dots&\dots&0&-1& m_{n-1}&-1
\\
0&\dots&\dots&0 &-1&2-z-\Theta
\end{pmatrix}
\begin{pmatrix}
f_1\\
f_2\\
\dots\\
\dots\\
\dots\\
f_n
\end{pmatrix}=\begin{pmatrix}
0\\
0\\
\dots\\
\dots\\
0\\
1
\end{pmatrix},
\end{equation}
where
 the diagonal elements $m_k$   $m_k$, $k=1, \dots, n-1$ of the system matrix are given by  (cf. \eqref{cfcf})
$$
m_k=2-z+v_k.
$$
Notice that 
$$
f_n=((J_{n-1}-z)^{-1}\delta_{n},\delta_{n}).
$$

The cofactor corresponding to the $(n, n)$-entry of the system matrix is clearly given by the polynomial
$q_{n-1}$ given by \eqref{poli} and its determinant equals $(2-z-\Theta)q_{n-1}-q_{n-2}=\frac1\Theta q_{n-1}-q_{n-2}\ne 0$
which yeilds the representation yields
$$
f_n=\frac{q_{n-1}}{\frac1\Theta q_{n-1}-q_{n-2}}.
$$
Therefore, 
$$
\text{det}_{J_n/J_{n-1}}(z)=1+v_n((J_{n-1}-z)^{-1}\delta_{n},\delta_{n})=1+v_n f_n.
$$

More explicitly, 
\begin{align*}
\text{det}_{J_n/J_{n-1}}&=1+v_nf_n=1+v_n\frac{ q_{n-1}}{\frac1\Theta q_{n-1}-q_{n-2}}
\\&=
\frac{(\frac1\Theta +\Theta +v_n)q_{n-1} -q_{n-2}-\Theta q_{n-1}}{\frac1\Theta q_{n-1}-q_{n-2}}
\\&
=
\frac{(2-z+v_n)q_{n-1} -q_{n-2}-\Theta q_{n-1}}{\frac1\Theta q_{n-1}-q_{n-2}}.
\end{align*}
Since by \eqref{recq},
$$
q_n=(2-z+v_n)q_{n-1}-q_{n-2},
$$
it follows
\be\label{comb2}
\text{det}_{J_n/J_{n-1}}
=\Theta \frac{q_{n}- \Theta q_{n-1}}{q_{n-1}-\Theta q_{n-2}}, \quad n\ge 2.
\ee
Using the multiplicativity property for perturbation determinants.
$$
\text{det}_{J_n/J}=
\prod_{k=2}^n\text{det}_{J_k/J_{k-1}} \cdot\text{det}_{J_1/J},
$$
and combining \eqref{comb1} and \eqref{comb2}, we  obtain
\begin{align*}
\text{det}_{J_n/J}
=&\Theta^{n-1}\frac{q_n-\Theta q_{n-1}}
{q_1-\Theta q_0}\cdot \frac{ j_0^{(1)}(\Theta)}{j_0^{(0)}(\Theta)}
\\&=\Theta^{n-1}\frac{q_n-\Theta q_{n-1}}
{q_1-\Theta q_0}\cdot \frac{\Theta(q_1-\Theta q_0)}{2}\frac1{j_0^{(0)}(\Theta)}
\\&
=\frac{\Theta^n (q_n-\Theta q_{n-1})}{2j_0^{(0)}(\Theta)}
=\frac{j_0^{(n)}(\Theta)}{j_0^{(0)}(\Theta)}.
\end{align*}
This proves \eqref{vita} for $n>1$.

\end{proof}
\subsection{A lemma from geometric perturbations theory}

Recall  the following proposition from the geometric perturbation theory.

\begin{proposition}[\cite{PAMS}]\label{pams}
Assume that $A $ and $V$ are bounded self-adjoint
 operators on a
separable Hilbert space $H$.
 
  Suppose that the spectrum  of $A$ has a part 
  $\sigma$ separated
from the remainder of the  spectrum $\Sigma$ in the sense that
  $$
  \text{spec}(A)=\sigma\cup \Sigma, \quad \text{dist}(\sigma,\Sigma)=d>0
  $$

Introduce the orthogonal projections $ P = E_A(\sigma)$ and $Q = E_{A+V}(U_{d/2}(\sigma))$,
 where
$U_\varepsilon(\sigma)$, $\varepsilon> 0$,
 is the open $\varepsilon$-neighborhood of the set $\sigma$.
 Here $E_{A}(\Delta)$  and $E_{A+V}(\Delta)$
denote the spectral projections for operators $A$ and $A+V$, respectively, corresponding
to a Borel set $\Delta\subset \bbR$.

 Assume that 
 $$\|V\|<\frac12d
 $$
 and
 $$ \text{conv.hull}(\sigma)\cap\Sigma=\emptyset \quad \text{or}\quad
  \text{conv.hull}(\Sigma)\cap\sigma=\emptyset. 
 $$
 Then 
 $$
 \|P-Q\|<1.
 $$
\end{proposition}

\end{document}